\documentclass[11pt]{article}

\usepackage{amsthm,amsmath,amssymb}
\usepackage{appendix}

\usepackage[colorlinks=true,citecolor=black,linkcolor=black,urlcolor=blue]{hyperref}

\theoremstyle{plain}
\newtheorem{theorem}{Theorem}

\newtheorem{corollary}[theorem]{Corollary}

\theoremstyle{definition}

\theoremstyle{remark}

\newcommand{\eqn}[1]{(\ref{#1})}

\newcommand{\ga}{\gamma}

\newcommand{\et}{\eta}
\newcommand{\si}{\sigma}
\newcommand{\sE}{{\mathcal E}}
\newcommand{\sG}{{\mathcal G}}
\newcommand{\eeq}{\end{equation}}
\newcommand{\beql}[1]{\begin{equation}\label{#1}}

\begin{document}
\theoremstyle{plain}

\begin{center}
{\large\bf Analysis of the gift exchange problem } \\
\vspace*{+.2in}

Moa Apagodu\\
Virginia Commonwealth University\\
Richmond, VA 23284, U.S.A.\\
\texttt{ mapagodu@vcu.edu}\\
\ \\
David Applegate\\
Google, Inc.\\
76 Ninth Ave., New York, NY 10011, U.S.A.\\
\texttt{ david@bcda.us}\\
\ \\
N. J. A. Sloane\footnote{Corresponding author} \\
The OEIS Foundation Inc.\\
11 So. Adelaide Ave., Highland Park, NJ 08904, U.S.A. \\
\texttt{ njasloane@gmail.com}\\
\ \\
Doron Zeilberger\\
Rutgers University (New Brunswick)\\
Piscataway, NJ 08854, U.S.A.\\
\texttt{ zeilberg@math.rutgers.edu}\\
\ \\

\vspace{0.1in}

{\bf Abstract}
\end{center}


In the {\em gift exchange} game
there are $n$ players and $n$ wrapped gifts.
When a player's number is called, that person
can either choose one of the remaining wrapped gifts,
or can ``steal'' a gift from someone who has already unwrapped it, subject
to the restriction that no gift can be stolen more than a total
of $\si$ times.  The problem is to determine the number of ways
that the game can be played out, for given values of $\si$ and $n$.
Formulas and asymptotic expansions are given for these
numbers. 
This work was inspired in part by a 2005 remark by Robert A. Proctor in the {\em On-Line Encyclopedia of Integer Sequences}.

\bigskip\noindent \textbf{Keywords:}
gift swapping, set partitions, restricted Stirling numbers, Bessel polynomials, 
hypergeometric functions, Almkvist-Zeilberger algorithm, Wilf-Zeilberger summation

\bigskip\noindent This is a sequel to the earlier (electronically published) article \cite{AppSlo09},
differing from it in that
there are two additional authors and several new theorems, including
the resolution of most of the conjectures,
and the extensive tables have been omitted.

\vspace{0.1in}

\vspace{0.1in}

\noindent Mathematics Subject Classifications: 05A, 11B37, 33F10


\section{The problem}\label{Sec1}

The following game is sometimes played at parties.
A number $\si$ (typically $1$ or $2$) is fixed in advance.
Each of the $n$ guests brings a wrapped gift, the gifts are
placed on a table (this is the ``pool'' of gifts),
and slips of paper containing the numbers $1$ to $n$
are distributed randomly among the guests.
The host calls out the numbers $1$ through $n$ in order.

When the number you have been given is called, you can either choose one 
of the wrapped (and so unknown) gifts remaining
in the pool, or you can take (or ``steal'') a gift that some earlier
person has unwrapped, subject to the restriction that no gift can be
``stolen'' more than a total of $\si$ times.

If you choose a gift from the pool, you unwrap it and show it
to everyone.
If a person's gift is stolen from them, they immediately get
another turn, and can either take a gift from the pool, or can
steal someone else's gift, subject always to the limit
of $\si$ thefts per gift.
The game ends when someone takes the last ($n$th) gift.

The problem is to determine the number of possible ways
the game can be played out,
for given values of $\si$ and $n$.

For example, if $\si=1$ and $n=3$, with guests $A, B, C$ and 
gifts numbered $1$, $2$, $3$, there are 42 different scenarios,
as follows.  We write $XN$ to indicate that 
guest $X$ took gift $N$ -- it is always clear from the context
whether the gift was stolen or taken from the pool. 
Also, provided we multiply the final answer by 6, we
can assume that the gifts are taken from the pool in 
the order $1, 2, 3$.
There are then seven possibilities:
\begin{align}\label{Eq1}
{} & A1, B2, C3, \nonumber \\
{} & A1, B2, C1, A3, \nonumber \\
{} & A1, B2, C1, A2, B3, \nonumber \\
{} & A1, B2, C2, B3, \nonumber \\
{} & A1, B2, C2, B1, A3, \nonumber \\
{} & A1, B1, A2, C3, \nonumber \\
{} & A1, B1, A2, C2, A3, \nonumber \\
\end{align}
and so the final answer is $6 \cdot 7 = 42$.

If we continue to ignore the factor of $n!$ that arises from
the order in which the gifts are selected from the pool,
the numbers of scenarios for the cases $\si=1$ and $n=1,2,3,4,5$ are
$1,2,7,37,266$, respectively.

We noticed that these five terms matched the beginning of entry
A001515 in \cite{OEIS}, although indexed differently.
The $n$th term of A001515 is defined to be $y_n(1)$,
where $y_n(x)$ is a Bessel polynomial
\cite{Gross78, KF49, RCI},
and for $n=0,1,2,3,4$ the values are indeed $1,2,7,37,266$.
Although there was no suggestion of gift-exchanging in that entry,
one of the comments there (contributed by
Robert A. Proctor on April 18, 2005) 
mentioned set partitions and restricted Stirling numbers, 
and this was enough of a hint to lead us to a solution
of the general problem.

\vspace*{+.2in}
\noindent{\bf Outline of paper.} 
The connection with set partitions is discussed in \S\ref{Sec2},
where Theorem~\ref{th1}  expresses the general solution to
the problem---the function $G_{\si}(n)$---as a sum of restricted Stirling numbers of the second kind.
We find it convenient to introduce our own notation, $E_{\si}(n,k)$ (see \eqn{EqE}),
for the latter numbers.
Theorem \ref{th2} gives various properties of the $E_{\si}(n,k)$,
and Theorem \ref{th3bis} gives an integral representation for $G_{\si}(n)$ that will 
enable us to use the methods of \cite{AlmZeil90} and \cite{ApaZeil06} to find
recurrences.

We will obtain two different linear recurrences for $G_{\si}(n)$,
which we label Type C and Type D.
Type C is somewhat simpler, and has leading coefficient $1$, 
while Type D has lower order (order $\si+1$, in fact,
see Theorem \ref{thsigmaplus1}). 
To see the difference between the two types of recurrence for $G_2(n)$, $G_3(n)$, and $G_4(n)$,
compare \eqn{EqG2d} with \eqn{EqG2e},
\eqn{EqG3e} with \eqn{EqG3D},
and \eqn{EqG4e} with the expression in Table~\ref{TabG4D}
(the first of each pair being the Type C recurrence,
the second the Type D recurrence).

Sections \ref{Sec3}, \ref{Sec4}, and \ref{Sec5} deal with cases $\si=1$ (where
there are connections with Bessel polynomials and hypergeometric functions,
see Theorem \ref{th3}), $\si=2$ (where we introduce the two types of recurrence),
and $\si \ge 3$, respectively.

\vspace*{+.2in}
\noindent{\bf Comments on the rules.} 
(i) If $\si = 1$ then once a gift has been stolen it
can never be stolen again.

\noindent
(ii) If $\si = 2$, and someone steals your gift, 
then if you wish you may immediately steal it back (provided
you hadn't stolen it to begin with), and then
it cannot be stolen again. Retrieving a gift in this way, 
although permitted by a strict interpretation of the rules,
may be prohibited at real parties.

\noindent
(iii) A variation of the game allows the last player
to take {\em any} gift that has been unwrapped,
regardless of how many times it has already been stolen,
as an alternative to taking the last gift from the pool.
This case only requires minor modifications of
the analysis, and we will not consider it here.

\noindent
(iv) We also ignore the complications caused by the fact that 
you brought (and wrapped) one of the gifts yourself, and so are
presumably unlikely to choose it when your number is called.


\section{Connection with partitions of labeled sets}\label{Sec2}

Let $H_{\si}(n)$ be the number of scenarios with $n$ gifts and a limit
of $\si$ steals, for $\si \ge 0, n \ge 1$.
Then $H_{\si}(n)$ is a multiple of $n!$, and we write
$H_{\si}(n) = n! G_{\si}(n-1)$, where in $G_{\si}(n-1)$ we
assume that the gifts are taken from the pool in the order $1,2,\ldots,n$.
We write $n-1$ rather than $n$ as the argument of $G_{\si}$ because the
$n$th gift plays a special (and less important) role. This also
simplifies the statement of Theorem \ref{th1}.

In words, $G_{\si}(n)$ is the number of scenarios when there
are $n+1$ gifts, with a limit of $\si$ steals per gift, and 
the gifts are taken from the pool in the order $1,2,\ldots,n+1$. 

We recall that the Stirling number of the second kind,
$S_2(i,j)$, is the number of partitions of the labeled set $\{1,\ldots,i\}$ into $j$ parts
\cite{Comtet, GKP},
while for $h \ge 1$ the $h$-restricted Stirling number of the second kind,
$S_2^{(h)}(i,j)$, is the number of partitions of $\{1,\ldots,i\}$ into $j$ parts
of size at most $h$ \cite{ChSm1, ChSm2, ChSm3}.
It is understood that parts are always nonempty.

\begin{theorem} \label{th1}
For $\si \ge 0$ and $n \ge 0$,
\beql{Eq2}
G_{\si}(n) = \sum_{k=n}^{(\si+1)n} S_2^{(\si+1)}(k,n)\,.
\eeq
\end{theorem}

\begin{proof}
Equation \eqn{Eq2} is an assertion about $G_{\si}(n)$, so we are
now discussing scenarios where there are $n+1$ gifts.
For $\si = 0$, $H_0(n+1) = (n+1)!$, so $G_0(n) = 1$,
in agreement with $S_2^{(1)}(n,n) = 1$.

We may assume therefore that $\si \ge 1$.
Let an ``action'' refer to a player choosing a gift~$\ga$, either by
taking it from the pool or by stealing it from another player.  Since
we are now assuming that the gifts are taken from the pool in order,
$\ga$ determines both the player and whether the action was to take
a gift from the pool or to steal it from another player.  So the
scenario is fully specified
simply by the sequence of $\ga$ values,
recording which gift is chosen at each action.
For example, the scenarios in \eqn{Eq1} are represented
by the sequences
$123$, $1213$, $12123$, $1223$, $12213$, $1123$, $11223$.
Since the game ends as soon as the $(n+1)$st gift is selected, the number
of actions is at least $n+1$ and at most $(\si+1)n+1$.

The sequence of $\ga$ values is therefore a sequence of integers
from $\{1,\ldots,n+1\}$ which begins with $1$, ends with $n+1$,
where each number $i \in \{1,\ldots,n\}$ appears 
at least once and at most $\si+1$ times
and $n+1$ appears just once,
and in which the first $i$ can appear only after $i-1$ has appeared.
Conversely, any sequence with these
properties determines a unique scenario.

Let $k$ denote the length of the sequence with the last entry
(the unique $n+1$) deleted.
We map this shortened sequence to a partition of $[1,\ldots,k]$
into $n$ parts: the first part records the positions of the $1$'s,
the second part records the positions of the $2$'s, $\ldots$,
and the $n$th part records the positions of the $n$'s.
Continuing the example,
for the seven sequences above,
the values of $k$ and the corresponding partitions are as
shown in Table~\ref{Tab1}.

\begin{table}[htb]
\caption{Values of $k$ and partitions corresponding to
the scenarios in \eqn{Eq1}. }\label{Tab1}
$$
\begin{array}{llll}
\mbox{scenario} & \mbox{sequence} & k & \mbox{partition} \\
A1, B2, C3 &    123 &   2 & 1, 2 \\
A1, B2, C1, A3 &        1213 &  3 & 13, 2 \\
A1, B2, C1, A2, B3~ &    12123 & 4 & 13, 24 \\
A1, B2, C2, B3 &        1223 &  3 & 1, 23 \\
A1, B2, C2, B1, A3~ &    12213 & 4 & 14, 23 \\
A1, B1, A2, C3 &        1123 &  3 & 12, 3 \\
A1, B1, A2, C2, A3~ &    11223 & 4 & 12, 34 \\
\end{array}
$$
\end{table}

\noindent The number of such partitions is precisely $S_2^{(\si+1)}(k,n)$.
Since the mapping from sequences to partitions is completely reversible,
the desired result follows.
\end{proof}

\begin{corollary}\label{cor2}
$G_{\si}(n)$ is equal to the number of partitions
of $\{1,2,\ldots,k\}$, for $n \le k \le (\si+1)n$,
into $n$ parts, each of size at most $\si + 1$.
\end{corollary}

\vspace*{+.2in}
\noindent{\bf Remark.} The sums $B(i) := \sum_{j}^{} S_2(i,j)$ 
are the classical Bell numbers. The sums $\sum_{j}^{} S_2^{(h)}(i,j)$
also have a long history \cite{MMW, MW55}.
However, the sums
$\sum_{i}^{} S_2^{(h)}(i,j)$
mentioned in \eqn{Eq2}
(where we sum on the first index rather than the second) 
do not seem to have studied before.
Note that the limits in \eqn{Eq2}
are the natural limits on the summand $k$, and
could be omitted.

To simplify the notation, and to put the most important variable first, let
\beql{EqE}
E_{\si}(n,k) :=  S_2^{(\si+1)}(k,n)\,,
\eeq
for $\si \ge 0$, $n \ge 0$, $k \ge 0$.
In words, $E_{\si}(n,k)$ is the number of partitions of $\{1, \ldots, k\}$ into exactly
$n$ parts of sizes in the range $[1, \ldots, \si+1]$.

For $n \ge 0$, $E_{\si}(n,k)$ is nonzero only
for $n \le k \le (\si+1)n$.
To avoid having to worry about negative arguments,
we define $E_{\si}(n,k)$ to be zero if either $n$ or $k$ is negative.
Then the answer to our problem can be written as
\beql{Eq4}
G_{\si}(n) = \sum_{k=n}^{(\si+1)n} E_{\si}(n,k)\,.
\eeq

Stirling numbers of the second kind satisfy many different recurrences
and generating functions \cite[Chap.~V]{Comtet},
and to a lesser extent this is also true of $E_{\si}(n,k)$.
We begin with four general properties.

\begin{theorem}\label{th2}

(i) Suppose $\si \ge 1$, $n \ge 0$, $k \ge 0$. Then $E_{\si}(n,k) = 0$ for $k<n$ or $k>(\si+1)n$,
$E_{\si}(n,k)=1$ if $k=n$,
and otherwise, for $n < k \le (\si + 1)n$,
\beql{EqAAB}
E_{\si}(n,k)
= \sum_{i=0}^{\min\{\si,k-n\}} \binom{k-1}{i} E_{\si}(n-1,k-1-i) \,.
\eeq

(ii) For $\si \ge 0$, $n \ge 0$, $n \le k \le (\si + 1)n$,
\beql{EqAAA}
E_{\si}(n,k)
= \sum_{(\nu_1,\ldots,\nu_{\si+1})}
\frac{k!}{\nu_1! \nu_2! \ldots \nu_{\si+1}! \, 1!^{\nu_1} 2!^{\nu_2} \cdots (\si+1)!^{\nu_{\si+1}} } \,,
\eeq
where the sum is over all $(\si+1)$-tuples of
nonnegative integers $(\nu_1,\ldots,\nu_{\si+1})$ satisfying
\begin{eqnarray}\label{EqAAD}
\nu_1 + \nu_2 + \nu_3 \cdots + \nu_{\si+1} & = & n \,, \nonumber \\
\nu_1 + 2\nu_2 + 3\nu_3 \cdots + (\si+1)\nu_{\si+1} & = & k \,.
\end{eqnarray}

(iii) The numbers $E_{\si}(n,k)$ have the exponential generating functions
\beql{EqAAC1}
\sum_{k=n}^{(\si+1)n} E_{\si}(n,k) \frac{y^k}{k!}
=
\frac{1}{n!}  \left(y+\frac{y^2}{2!}+\cdots + \frac{y^{\si+1}}{(\si+1)!}\right)^n  \,,
\eeq
\beql{EqAAC2}
\sum_{n=0}^{\infty} \sum_{k=n}^{(\si+1)n} E_{\si}(n,k)x^n \frac{y^k}{k!}
=
\exp \left[ x\left(y+\frac{y^2}{2!}+\cdots + \frac{y^{\si+1}}{(\si+1)!}\right) \right] \,.
\eeq

(iv) Suppose $\si \ge 1$. Then $E_{\si}(n,k) = 0$ for $k<n$ or $k>(\si+1)n$,
$E_{\si}(n,k)=1$ if $k=n$,
and otherwise, for $n < k \le (\si + 1)n$,
\beql{EqAAE}
E_{\si}(n,k)
=  \frac{k!}{(k-n)} \sum_{i=1}^{\min\{\si,k-n\}} \frac{(n+1)i-k+n}{(i+1)! (k-i)!} E_{\si}(n,k-i) \,.
\eeq
\end{theorem}

\begin{proof}
(i) This is an analog of the ``vertical'' recurrence
for the Stirling numbers
\cite[Eq.~$\lbrack$3c$\rbrack$,~p.~209]{Comtet}
(note that because \eqn{EqE} involves a transposition, 
``horizontal'' recurrences for $E_{\si}$ correspond to ``vertical'' 
recurrences for Stirling numbers).
The idea of the proof is to take a partition of $[1,\ldots,k]$ into $n$ parts,
remove the part containing $k$, and renumber the
remaining parts.  If the part containing $k$ has size $i+1$,
$0 \le i \le \si$, there are $\binom{k-1}{i}$ possibilities
for the other elements in that part.
(ii) This follows by standard counting 
arguments (cf. \cite[Th.~B,~p.~205]{Comtet}), taking 
$\nu_i$ to be the number of parts of size~$i$ in the partition.
(iii) \eqn{EqAAC1} and \eqn{EqAAC2} are  analogs of the ``vertical'' generating functions
for the Stirling numbers \cite[Eqs.~$\lbrack$2a,2b$\rbrack$,~p.~206]{Comtet},
and follow directly from (i).
(iv) We let $\eta = k-n$ (the ``excess'' of $k$ over $n$), and rewrite \eqn{EqAAC1} as
\beql{EqAAC3}
\sum_{\eta = 0}^{\si n} \frac{n!}{(n+\eta)!}  E_{\si}(n,n+\eta) y^{\eta}
=
 \left(1+\frac{y}{2!}+\cdots + \frac{y^{\si}}{(\si+1)!}\right)^n \,.
\eeq
We now apply the J. C. P. Miller method for exponentiating a polynomial.
In the notation of \cite{Zeil95}, we take 
\beql{EqAAC4}
P(y) =  \left(1+\frac{y}{2!}+\cdots + \frac{y^{\si}}{(\si+1)!}\right) \,,  \quad 
a(n,\eta) = \frac{n!}{(n+\eta)!}  E_{\si}(n,n+\eta)\,.
\eeq
Then \eqn{EqAAE} follows at once from \eqn{EqAAC3} and the main theorem of \cite{Zeil95}.
\end{proof}

From \eqn{Eq4} and  \eqn{EqAAA} we have:

\beql{EqGsa}
G_{\si}(n) ~=~ \sum_{k=n}^{(\si+1)n} \sum_{(\nu_1,\ldots,\nu_{\si+1})}
\frac{k!}{\nu_1! \nu_2! \ldots \nu_{\si+1}! \, 1!^{\nu_1} 2!^{\nu_2} \cdots (\si+1)!^{\nu_{\si+1}} } \,,
\eeq
where the inner sum is over all $(\si+1)$-tuples of
nonnegative integers $(\nu_1,\ldots,\nu_{\si+1})$ satisfying \eqn{EqAAD}.
This may be rewritten as a sum of multinomial coefficients:
\beql{EqGsb}
G_{\si}(n) ~=~
\frac{1}{n!}~
\sum_{i_1=1}^{\si+1}
\sum_{i_2=1}^{\si+1}
\cdots
\sum_{i_n=1}^{\si+1}
\genfrac{(}{)}{0pt}{0}{i_1+i_2+\cdots+i_{n}}{i_1,~i_2,~\cdots,~i_{n}} \,,
\eeq
where $i_r$ is the size of the $r$th part,
or, equivalently, the number of times the $r$th gift is chosen in an action.

Equation \eqn{EqAAA} also leads to an integral representation 
and generating function for $G_{\si}(n)$.
\begin{theorem}\label{th3bis}
(i) For $\si \ge 1$,
\beql{EqAAF}
G_{\si}(n) ~=~ \frac{1}{n!} \int_{0}^{\infty} e^{-y}
    \left( \sum_{j=1}^{\si+1} \frac{y^j} {j!} \right)^n dy \,.
\eeq
(ii) The $G_{\si}(n)$ have ordinary generating function
\begin{eqnarray}\label{EqAAFg}
\sG _{\si} (x) & = & \sum_{n=0}^{\infty} G_{\si}(n) x^n \nonumber \\
       & = & \int_{0}^{\infty} exp \left( -y +
 x  \sum_{j=1}^{\si+1} \frac{y^j} {j!} \right) dy \,.
\end{eqnarray}

\end{theorem}
\begin{proof}
(i) Using Euler's integral representation for the factorial,
\beql{EqAAG}
k! ~=~ \int_{0}^{\infty} e^{-y} y^k dy \,,
\eeq
the right-hand side of \eqn{EqAAF} is
\begin{eqnarray} \label{EqAAH}
 &  & \frac{1}{n!} \int_{0}^{\infty} e^{-y}
 \sum_{(\nu_1,\ldots,\nu_{\si+1})}
 \frac{n! ~ y^{\nu_1 + 2\nu_2 + 3\nu_3 + \cdots + (\si+1)\nu_{\si+1}}} 
 {\nu_1! \nu_2! \ldots \nu_{\si+1}! \,  2!^{\nu_2} \cdots (\si+1)!^{\nu_{\si+1}} } 
 \nonumber \\
       & = & 
       \sum_{(\nu_1,\ldots,\nu_{\si+1})}
       \frac{ (\nu_1 + 2\nu_2 + \cdots + (\si+1)\nu_{\si+1})! }
        {\nu_1! \nu_2! \ldots \nu_{\si+1}! \,  2!^{\nu_2} \cdots (\si+1)!^{\nu_{\si+1}}}\,, 
\end{eqnarray}
where the sums are over all $(\si+1)$-tuples of
nonnegative integers $(\nu_1,\ldots,\nu_{\si+1})$ satisfying
$\nu_1 + \nu_2 +  \cdots + \nu_{\si+1}  =  n$.
This is the same as the expression for $G_{\si}(n)$ in \eqn{EqGsa}.
(ii) is an immediate consequence of (i).
\end{proof}

The recurrence \eqn{EqAAB} makes it easy to  
compute as many values of $E_{\si}(n,k)$ and $G_{\si}(n)$ as one wishes.
In the rest of the paper we will derive further formulas and recurrences,
and asymptotic estimates for these numbers.


\section{The case \texorpdfstring{$\si = 1$}{sigma = 1}}\label{Sec3}

In the case $\si = 1$, i.e., when a gift can be stolen at most once, from Theorem \ref{th2}
we have $E_1(1,k) = \delta_{1,k}$, and, for $n \ge 2$,
$E_1(n,k)=0$ for $k<n$ and $k>2n$, $E_1(n,n)=1$, and 
\beql{EqE1a}
E_1(n,k) = E_1(n-1,k-1) + (k-1)E_1(n-1,k-2) \,,
\eeq
for $n < k \le 2n$.
We also have the explicit formula
\beql{EqE1b}
E_1(n,k) = \frac{k!}{(2n-k)!~(k-n)!~2^{k-n}} \,,
\eeq
for $n \le k \le 2n$; and the generating function
\beql{EqE1c}
\sum_{n=0}^{\infty} \sum_{k=n}^{2n}~E_{1}(n,k)~x^n \frac{y^k}{k!}
=
e^{x(y+y^2/2)} \,.
\eeq
It follows from \eqn{Eq4} and \eqn{EqE1b} that
\begin{eqnarray}\label{EqG1a}
G_1(n) & = & \sum_{k=n}^{2n} \frac{k!}{(2n-k)!~(k-n)!~2^{k-n}} \nonumber \\
       & = & \sum_{i=0}^{n} \frac{(n+i)!}{(n-i)!~i!~2^i}  \,.
\end{eqnarray}

Equation \eqn{EqG1a} shows that the sequence $G_1(n)$ is indeed given by 
entry A001515 in \cite{OEIS}.
That entry states (mostly without proof) several other properties of these numbers,
taken from various sources, notably Grosswald \cite{Gross78}.
We collect some of these properties in the next theorem.
We recall from \cite{Gross78} that the Bessel polynomial $y_n(z)$ is given
by
\beql{EqBess1}
y_n(z) := \sum_{i=0}^{n} \frac{(n+i)!z^i}{(n-i)!~i!~2^i}  \,.
\eeq
Also ${}_2F_{0}$ and (later) ${}_2F_{1}$ denote hypergeometric functions.
\begin{theorem}\label{th3}
(i) 
\beql{EqG1b}
G_1(n) = y_n(1)\,.
\eeq
(ii) 
\beql{EqG1c}
G_1(n) = {}_2F_{0}\left[ \begin{array}{c}
                            n+1,-n \\
                             -
                            \end{array}
                             ;
                            \begin{array}{c}
                            -\frac{1}{2}
                            \end{array}
                            \right] \,.
\eeq
(iii)
\beql{EqG1d}
G_1(n) = (2n-1)G_1(n-1) + G_1(n-2) \,.
\eeq
for $n \ge 2$, with $G_1(0)=1, G_1(1)=2$.

\noindent
(iv) The $G_1(n)$ have exponential generating function
\beql{EqG1e}
\sE\sG_1(x)  ~:=~ \sum_{n=0}^{\infty} G_1(n)\frac{x^n}{n!} ~=~ \frac{ e^{1-\sqrt{1-2x}}}{\sqrt{1-2x}} \,.
\eeq
(v)
\beql{EqG1f}
G_1(n) ~ \sim ~ \frac{e(2n)!}{n! 2^n} \mbox{~as~} n \rightarrow \infty \,. 
\eeq
\end{theorem}

\begin{proof}
(i) and (ii) are immediate consequences of \eqn{EqG1a}. 
(iii) It is easy to verify from \eqn{EqE1b} that  
\beql{EqE1d}
E_1(n,k) = (2n-1)E_1(n-1,k-2) + E_1(n-2,k-2)\,.
\eeq
Our conventions about negative arguments make it 
unnecessary to put any restrictions on the range over which \eqn{EqE1d}
holds. By summing \eqn{EqE1d} on $k$ we obtain \eqn{EqG1d}.
(Equation \eqn{EqG1d} also follows from one of the recurrences for Bessel polynomials \cite[Eq.~(7),~p.~18]{Gross78}, \cite{KF49}.)
(iv)  By multiplying \eqn{EqG1d} by $x^n/n!$ and summing on $n$ from $2$
to $\infty$ we obtain the differential equation
\beql{EqG1i}
\sE \sG_1''(x) = 3\,  \sE \sG_1'(x) + 2x \, \sE \sG_1''(x) + \sE\sG_1(x)\,.
\eeq
Then the right-hand side of \eqn{EqG1e} is the unique solution of \eqn{EqG1i}
which satisfies $\sE\sG_1(0) = 1$, $\sE\sG_1'(0) = 2$.
(v) This follows from \eqn{EqG1a}, since the terms $i=n-1$
and $i=n$ dominate the sum (see also \cite[Eq.~(1),~p.~124]{Gross78}).
\end{proof}


\section{The case \texorpdfstring{$\si = 2$}{sigma = 2}}\label{Sec4}

In the case when $\si = 1$, i.e., when a gift can be stolen at most once, the problem, as we saw in the
previous section, turned out to be related to the values of Bessel polynomials,
and the principal sequence, $G_1(n)$, had been studied before.
For $\si \ge 2$, we appear to be in new territory---for one thing,
the sequences $G_2(n), G_3(n), \ldots$ were not in \cite{OEIS}.\footnote{$G_2(n), G_3(n), G_4(n)$ are now 
entries A144416, A144508, A144509 in \cite{OEIS}.}

We naturally tried to find analogs of the various parts of Theorem \ref{th3}
that would hold for $\si \ge 2$.
Let us begin with the simplest result, the asymptotic behavior.
This is directly analogous to Theorem \ref{th3}(v).

\begin{theorem}\label{th4}
For fixed $\si \ge 1$,
\beql{EqGsf}
G_{\si}(n) ~ \sim ~  \frac{e((\si+1)n)!}{n! {(\si+1)!}^n} \mbox{~as~} n \rightarrow \infty \,.
\eeq
\end{theorem}

\begin{proof} (Sketch)
The two terms of \eqn{EqGsa} corresponding to
$ \{ k=(\si+1)n, \nu_{\si+1}=n$, other $\nu_i=0 \} $ and
$ \{ k=(\si+1)n-1, \nu_{\si+1}=n-1, \nu_{\si}=1$, other $\nu_i=0 \} $
dominate the right-hand side of \eqn{EqGsa},
and are both equal to 
$((\si+1)n)!/(n! {(\si+1)!}^n)$.
Dividing the sum by this quantity gives a converging sum,
in which a subset of terms approach $1+1+1/2!+1/3!+...$,
while the others vanish as $n \rightarrow \infty$.
\end{proof}

Concerning Theorem \ref{th3}(i), we do not know if there is a generalization 
of Bessel polynomials whose value gives \eqn{EqGsa} for $\si \ge 2$.

As for Theorem \ref{th3}(ii), there is a relationship with hypergeometric
functions in the case $\si = 2$.  From \eqn{EqAAA} we have
\begin{eqnarray}\label{EqE2a}
E_2(n,k) & = & 
\sum_{c=\max\{0,k-2n\}}^{\lfloor(k-n)/2\rfloor}
\frac{k!}
{(2n-k+c)! (k-n-2c)! c! \, 2^{k-n-c} 3^c} \nonumber \\
& = &
\sum_{c=\max\{0,\et-n\}}^{\lfloor \et/2\rfloor}
\frac{k!}
{(n-\et+c)! (\et-2c)! c! \, 2^{\et-c} 3^c} \,,
\end{eqnarray}
where $\et=k-n$.

\begin{theorem}\label{th5}
(i) Let $\et=k-n$.

\noindent
If $\et \le n$ then
\beql{EqE2b}
E_2(n,k) =
\frac{(n+\et)!}{\et! (n-\et)! 2^{\et} } ~ 
   {}_2F_{1}\left[\begin{array}{c}
                  -\et/2,-\et/2+1/2 \\
                       n-\et+1
                  \end{array}
                   ;
                  \begin{array}{c}
                  \frac{8}{3}
                  \end{array}
                  \right] \,.
\eeq

\noindent
If $\et \ge n$ then
\beql{EqE2c}
E_2(n,k) =
\frac{(\et+n)!}{(2n-\et)! (\et-n)! 2^{n} 3^{\et-n} } ~ 
  {}_2F_{1}\left[ \begin{array}{c}
                  -n+\et/2,-n+\et/2+1/2 \\
                       \et-n+1
                  \end{array}
                   ;
                  \begin{array}{c}
                  \frac{8}{3}
                  \end{array}
                  \right] \,.
\eeq

\noindent
(ii)
\begin{eqnarray}\label{EqG2c}
G_2(n) & = &
\sum_{ \et = 0 }^{n-1} ~
\frac{(n+\et)!}{\et! (n-\et)! 2^{\et} } ~ 
  {}_2F_{1}\left[ \begin{array}{c}
                  -\et/2,-\et/2+1/2 \\
                       n-\et+1
                  \end{array}
                   ;
                  \begin{array}{c}
                  \frac{8}{3}
                  \end{array}
                  \right]  \nonumber \\
& + & \sum_{ \et = n }^{2n} ~
\frac{(n+\et)!}{(2n-\et)! (\et-n)! 2^{n} 3^{\et-n} } ~ 
  {}_2F_{1}\left[ \begin{array}{c}
                  -n+\et/2,-n+\et/2+1/2 \\
                       \et-n+1
                  \end{array}
                   ;
                  \begin{array}{c}
                  \frac{8}{3}
                  \end{array}
                  \right].
\end{eqnarray}
\end{theorem}

\begin{proof}
(i) follows from \eqn{EqE2a} using the standard rules for converting sums of
products of factorials to hypergeometric functions (cf. \cite{And74}),
and (ii) follows from \eqn{Eq4}.
\end{proof}

We next give an analog of \eqn{EqE1d} and our first (Type C) recurrence for
$G_2(n)$.

\begin{theorem}\label{th6}
(i)
\begin{align}\label{EqE2d}
 E_2(n,k) &  = (9 n^2 - 9 n + 2) E_2(n-1,k-3)/2
       - 5 E_2(n-1,k-1)/2 \nonumber \\
& +\, (9 n^2 - 36 n + 35) E_2(n-2,k-4)/2
       + 6(n-1) E_2(n-2,k-3)\nonumber \\
& -\, 3 E_2(n-2,k-2)/2 
        +\, 3(2 n-5) E_2(n-3,k-4)
       + 5 E_2(n-3,k-3)/2 \nonumber \\
 &      +\, 5 E_2(n-4,k-4)/2 \, .
\end{align}

\noindent
(ii) 
\begin{align}\label{EqG2d}
G_2(n) & = (9 n^2 - 9 n - 3) G_2(n-1)/2 \nonumber \\
& +\, (9 n^2 - 24 n + 20) G_2(n-2)/2 \nonumber \\
& +\, (6 n - 25/2) G_2(n - 3) + 5 G_2(n - 4)/2 \, ,
\end{align}
for $n \ge 4$, with $G_2(0)=1$, $G_2(1) = 3$,
$G_2(2) = 31$, $G_2(3) = 18252$.
\end{theorem}

\begin{proof}
(i) Equation \eqn{EqE2d} follows by applying the Almkvist-Zeilberger technique of
``differentiating under the integral sign'' from \cite{AlmZeil90},
starting with the integral
representation for $G_{2}(n)$ given in \eqn{EqAAF}.
(ii) Eq. \eqn{EqG2d} follows by summing \eqn{EqE2d} on $k$, just as \eqn{EqG1d} followed from \eqn{EqE1d}.
\end{proof}

\vspace*{+.2in}
\noindent{\bf Remarks.}
(i) To apply the technique of ``differentiating under the integral sign'' from \cite{AlmZeil90}, we used the Maple implementation given in \cite{AlmZeil91}.

(ii) The recurrence \eqn{EqE2d} can also be proved
using \eqn{EqE2b}, \eqn{EqE2c}, and
Gauss's contiguity relations for hypergeometric functions
\cite[\S2.1.2]{Erd}, \cite[\S14.7]{WW}.

(iii) A third proof may be obtained using
the method of Sister Mary Celine Fasenmyer,
as described in \S4.1 of \cite{AeqB}.

\vspace*{+.2in}
We discovered \eqn{EqG2d} by experiment, using Theorem
\ref{th4} to suggest the leading term. 
(If $r(n)$ denotes the right-hand side of \eqn{EqGsf},
then $r(n)/r(n-1) = (9 n^2 - 9 n + 2)/2$.)
This is a fourth-order recurrence for $G_2(n)$.
We also discovered (again by experimenting) a third-order recurrence
for $G_2(n)$, although the coefficient of
$G_2(n)$ on the left side is now $n-2$ rather than 1.
\begin{theorem}\label{th39}
For $n \ge 3$,
\begin{align}\label{EqG2e}
(n-2) G_2(n) & =  n (9 n^2-27 n+17) G_2(n-1)/2 \nonumber \\
 & +  (6 n^2-15 n+13/2) G_2(n-2) \nonumber \\
 & +  (5 n-5) G_2(n-3)/2 \, ,
\end{align}
with $G_2(0)=1$, $G_2(1) = 3$,
$G_2(2) = 31$.
\end{theorem}
In view of \eqn{EqG2c}, \eqn{EqG2e} is equivalent to a complicated
identity involving hypergeometric functions.
\begin{proof}
To prove \eqn{EqG2e} we apply the Almkvist-Zeilberger technique of ``differencing under
the integral sign'' from \cite[\S7]{AlmZeil90},
again starting from the integral 
representation for $G_{2}(n)$ given
in \eqn{EqAAF}.
(We again used the Maple implementation of the technique from \cite{AlmZeil91}.)
\end{proof}

Since \eqn{EqE2d} uses the {\em continuous}
version of the algorithm in \cite{AlmZeil90},
and \eqn{EqG2e} the {\em discrete} version,
we refer to these recurrences as Types C  and D respectively.
When $\si=1$, both Types C and D produce \eqn{EqG1d}.

Finally, if we need more terms of the asymptotic expansion than 
are given in Theorem \ref{th4}
(or a computer-certified proof of it!),
we can apply the Poincar\'{e}--Birkhoff--Trjitzinsky method
as presented in \cite{WimpZeil85} to the 
Type C recurrence \eqn{EqG2d}.
This works for any $\si$; for $\si=2$, we find that
\beql{EqAsympt2}
G_2(n) ~ \sim ~ \frac{e\,(3n)!}{n! \, 6^n}
\, \Big(1 ~+~ \frac{1}{3 \, n} ~+~ \frac{1}{54 \, n^2}  - \frac{8}{81 \, n^3} - \cdots \Big) \,.
\eeq


\section{The cases \texorpdfstring{$\si \ge 3$}{sigma at least 3}}\label{Sec5}

In this final section we discuss the cases $\si \ge 3$, mostly concentrating
on $G_{\si}(n)$ rather than $E_{\si}(n,k)$.
We have not found any connections between these
$G_{\si}(n)$ and generalized Bessel polynomials or hypergeometric functions.

Let us first prove that recurrences for $G_{\si}(n)$ and $E_{\si}(n,k)$ always exist.
This follows from the Wilf--Zeilberger ``Fundamental Theorem
for multivariate sums'' (\cite[Theorem~4.5.1]{AeqB}, \cite{WZ92a}).

\begin{theorem}\label{th7}
\noindent
(i) For $\si \ge 1$,
there is a number $\delta \ge 0$ such that
$E_{\si}(n,k)$ satisfies a recurrence of the form
\beql{EqEsWZ}
\sum_{i=0}^{\delta} \sum_{j=0}^{\delta} C_{i,j}^{(E)}(n) E_{\si}(n-i,k-j) =0
\mbox{~for~all~} n \ge \delta \,,
\eeq
where the coefficients $C_{i,j}^{(E)}(n)$ are polynomials
in $n$ with coefficients depending on $i$ and~$j$.

\noindent
(ii) For $\si \ge 1$,
there is a number $\delta \ge 0$ such that
$G_{\si}(n)$ satisfies a recurrence of the form
\beql{EqGsWZ}
\sum_{i=0}^{\delta} C_{i}^{(G)}(n) G_{\si}(n-i)=0
\mbox{~for~all~} n \ge \delta \,,
\eeq
where the coefficients $C_{i}^{(G)}(n)$ are polynomials
in $n$ with coefficients depending on $i$.
\end{theorem}

\begin{proof}
(ii) As usual, Eq. \eqn{EqGsWZ} follows by summing \eqn{EqEsWZ} on $k$.
(i) We will use the case $\si = 3$ to illustrate the proof,
the general case being similar. We know from \eqn{EqAAA} that
\beql{EqAAA3}
E_{3}(n,k)
= \sum_{a,b,c,d}
\frac{k!}{a! \, b! \, c! \, d! \, 2^b \, 6^c \, 24^d } \,,
\eeq
where the sum is over all $4$-tuples of
nonnegative integers $(a,b,c,d)$ satisfying
\begin{eqnarray}
a + b + c + d & = & n \,, \nonumber \\
a + 2b + 3c + 4d & = & k \,. \nonumber
\end{eqnarray}
In other words,
\beql{EqAAA4}
\frac{E_{3}(n,k)}{2^n}
= \sum_{c,d}
\frac{k!}{(2n-k+c+2d)! \, (k-n-2c-3d)! \, c! \, d! \, 2^{k-c} \, 3^{c+d} } \,,
\eeq
where now the sum is over all values of $c$ and $d$ for
which the summand is defined.
This summand is a ``holonomic proper-hypergeometric term'',
in the sense of \cite{WZ92a}, and it follows from
the Fundamental Theorem in that paper that 
$E_{3}(n,k)/2^n$ and hence $E_{3}(n,k)$
satisfies a recurrence of the desired form.
Similarly, in the general case, we write the summand in 
$E_{\si}(n,k)$ as a function of $n, k, a_3, \ldots, a_{\si+1}$,
again obtaining a holonomic proper-hypergeometric term.
\end{proof}

The Type C recurrences for $G_3(n)$ and $G_4(n)$ 
(obtained as in \S\ref{Sec4}) are:\footnote{The Types C and D
recurrences for all of $G_1(n)$ through $G_{15}(n)$ are given in two on-line appendices
to this article,
{\tt http://www.math.rutgers.edu/$\sim$zeilberg/tokhniot/oGifts2.txt}
and  
{\tt http://www.math.rutgers.edu/$\sim$zeilberg/tokhniot/oGifts1.txt} respectively.
}

\begin{eqnarray}\label{EqG3e}
 G_3(n)  & = & (32 n^3/3 - 16 n^2 + 10 n/3 - 49/6) G_3(n-1) \nonumber \\
       & + & (48 n^3 - 236 n^2 + 1157 n/3 - 650/3) G_3(n-2) \nonumber \\
       & + & (80 n^3 - 382 n^2 + 641 n - 511) G_3(n-3)/3 \nonumber \\
       & + & (64 n^3/3 - 218 n^2 + 2696 n/3 - 7915/6) G_3(n-4) \nonumber \\
       & + & (56 n^2 - 490 n + 6853/6) G_3(n-5) \nonumber \\
       & + & (56 n - 1703/6) G_3(n-6) \nonumber \\
       & + & 58 G_3(n-7)/3 \,, 
\end{eqnarray}

\footnotesize
\begin{eqnarray}\label{EqG4e}
G_4(n) & = &   (625\,{n}^{4}-1250\,{n}^{3}+625\,{n}^{2}-300\,n-543) G_4(n-1)/24 \nonumber \\
& + & (27500\,{n}^{4}-184000\,{n}^{3}+447500\,{n}^{2}-473075\,n+180003) G_4(n-2) /72 \nonumber \\
& + & (336875\,{n}^{4}-2546500\,{n}^{3}+7679675\,{n}^{2}-12016800\,n+8048577) G_4(n-3)/864 \nonumber \\
& + & (4833125\,{n}^{4}-77581625\,{n}^{3}+476892700\,{n}^{2}-1304291160\,n+1325759504) G_4(n-4)/2592 \nonumber \\
& + & (1700625\,{n}^{4}+28316750\,{n}^{3}-605973450\,{n}^{2}+3123850885\,n-5033477363) G_4(n-5)/7776 \nonumber \\
& + & (2670000\,{n}^{4}-64380500\,{n}^{3}+704577200\,{n}^{2}-3610058445\,n+6818722190)  G_4(n-6)/7776 \nonumber \\
& + & (2002500\,{n}^{4}-51976000\,{n}^{3}+517392050\,{n}^{2}-2252744530\,n+3561765885) G_4(n-7)/7776 \nonumber \\
& + & (9078000\,{n}^{3}-209915400\,{n}^{2}+1640828980\,n-4301927039) G_4(n-8)/7776 \nonumber \\
& + & (5393400\,{n}^{2}-91413680\,n+390747263) G_4(n-9)/2592 \nonumber \\
& + & (1593990\,n-14522219) G_4(n-11)/972 \nonumber \\
& + & 310343 G_4(n-11)/648 \,.
\end{eqnarray}
\normalsize

We also found analogous recurrences for the $E_{\si}(n,k)$, which we omit.

The Type D recurrence for $G_3(n)$ is given in \eqn{EqG3D},
and that for $G_4(n)$ can be seen in Table \ref{TabG4D}. 
 
\begin{eqnarray}\label{EqG3D}
  && 3\,(64\,{n}^{3}-360\,{n}^{2}+762\,n-547)\, G_3 \left( n \right) ~= \nonumber \\
   && \left( 2048\,{n}^{6}-14592\,{n}^{5}+42304\,{n}^{4}-58384\,{n}^{3}  + 36972\,{n}^{2}-10888\,n+2381 \right) G_{{3}} \left( n-1 \right) \nonumber \\
  &&+ \left( 5376\,{n}^{5}-35616\,{n}^{4}+92200\,{n}^{3}-110788\,{n}^{2} + 54186\,n-5365 \right) G_{{3}} \left( n-2 \right)  \nonumber \\
  && +\left( 5376\,{n}^{4}-27552\,{n}^{3}+52616\,{n}^{2}-45620\,n+10514 \right) G_{{3}}
 \left( n-3 \right) \nonumber \\
   &&+ \left( 1856\,{n}^{3}-4872\,{n}^{2}+6786\,n-
2349 \right) G_{{3}} \left( n-4 \right) 
\end{eqnarray}

\begin{table}[htb]
\caption{The Type D recurrence for $G_4(n)$..}
\label{TabG4D}

\footnotesize
\begin{align}
  & 72 \, (16687500\,{n}^{6}-209150000\,{n}^{5}+1070031875\,{n}^{4}-3019737375
\,{n}^{3}+4945130775\,{n}^{2} \nonumber \\
  & -4329975510\,n+1513065336) \,  G_4 \left( n \right) ~= \nonumber \\
 & \nonumber \\
 &  ( 31289062500\,{n}^{10}-454734375000\,{n}^{9}+2821911328125\,{n}^{8}-10081802109375\,{n}^{7} \nonumber \\
  & +22781118187500\,{n}^{6}-33185759803125\,{n}^{5}+30632133843750\,{n}^{4}-17235043672875\,{n}^{3} \nonumber \\
  & + 5483042423925\,{n}^{2}-700627863570\,n-57348303408 )\, G_{{4}} \left( n-1 \right)  \nonumber \\
 & \nonumber \\
 & + ( 141843750000\,{n}^{9}-1990540625000\, {n}^{8}+11724386562500\,{n}^{7}  -39078979093750\,{n}^{6} \nonumber \\
 & + 81505745228125\,{n}^{5}-107513140175625\,{n}^{4}+84513872351000\,{n}^{3}-33225357802500\,{n}^{2} \nonumber \\
 & +2737777538500\,n+1197797898465 ) \, G_{{4}} \left( n-2 \right) \nonumber \\
 & \nonumber \\ 
 & + ( 252815625000\,{n}^{8}-3168622500000\,{n}^{7}+16127100406250\,{n}^{6}-45548278450000\,{n}^{5} \nonumber \\
 & +80090937641250\,{n}^{4}-86115353337125\,{n}^{3}+47445915625400\,{n}^{2}-6693899844450\,n \nonumber \\
 & -2609871946015 ) \, G_{{4}} \left( n-3 \right) \nonumber \\
 & \nonumber \\
  & + ( 199248750000\,{n}^{7}-1999129125000\,{n}^{6}+
7757225837500\,{n}^{5}-16990061751250\,{n}^{4} \nonumber \\
 & +23960112482875\,{n}^{3}-17664322875275\,{n}^{2}+4396729093865\,n+802753105180 ) \, G_{
{4}} \left( n-4 \right) \nonumber \\ 
 & \nonumber \\
 &  + ( 58189312500\,{n}^{6}-380170175000\, {n}^{5}+957510585625\,{n}^{4}-1734293884125\,{n}^{3} \nonumber \\
 & + 1621184408800\,{n}^{2}-573345040895\,n-48634580313 ) \, G_{{4}} \left( n-5 \right) \,. \nonumber
\end{align}
\normalsize
\end{table}

Equation \eqn{EqG3D} and the expression in Table~\ref{TabG4D}  are linear recurrences of orders $4$ and $5$,
respectively, and the Type D recurrences for $G_5(n)$ and $G_6(n)$ are similarly of orders $6$ and $7$.
An analogous property holds for all $\si$.

\begin{theorem}\label{thsigmaplus1}
For $\si \ge 1$, the Type D recurrence for $G_{\si}(n)$ has order at most $\si +1$.
\end{theorem}
\begin{proof}
This follows from Theorem~AZ of \cite{ApaZeil06}.
To see this we apply that theorem to the quantity
$$
\widehat{G_{\si}}(n) ~:=~ n! \, G_{\si}(n) ~=~
\int_{0}^{\infty} e^{-x}
    \left( \sum_{j=1}^{\si+1} \frac{x^j} {j!} \right)^n dx \,.
$$
That is, for the expression $F(n,x)$ in the statement of
Theorem AZ we take $POL(n,x)=1$, $a(x)=-x$, $b(x)=1$, $S_p(x)=1$,
$s(x) = \sum_{j=1}^{\si+1} x^j/j!$, and $t(x)=1$. The conclusion of the theorem is that
$\widehat{G_{\si}}(n)$ (and hence $G_{\si}(n)$) satisfies a linear recurrence of order at most $L = \si+1$. 
\end{proof}

\end{document}